\documentclass[11pt]{article}

\usepackage{algorithm, algpseudocode}
\usepackage{amsmath,amssymb,amsthm}
\usepackage{mathrsfs}
\usepackage{undertilde} 
\usepackage{subfig}
\usepackage{color} 
\usepackage{graphicx}
\usepackage{makecell}
\usepackage{tabularx}
\usepackage{multirow}
\usepackage{multicol}
\usepackage{arydshln}
\usepackage{epstopdf}
 \vspace{0.6cm}

\newtheorem{lemma}{Lemma}[section]

\newtheorem{them}[lemma]{Theorem}
\newtheorem{lm}[lemma]{Lemma}

\begin{document}

\title{Spectral characterization of the complete graph removing a path of small length\footnote{This work is supported by the National Natural Science Foundation of China (No. 11471005) }}
\author{\small $^a$Lihuan Mao \qquad $^b$Sebastian M. Cioab\u{a}\qquad$^a$Wei Wang \footnote{The corresponding author. E-mail address:~wang$\_$weiw@xjtu.edu.cn}\\
\small $^a$School of Mathematics and Statistics, Xi'an Jiaotong University, Xi'an, P.R. China, 710049\\
\small $^b$ Department of Mathematical Sciences, University of Delaware, Newark, DE 19716-2553, USA}
\date{}
 \maketitle

\abstract
A graph $G$ is said to be \emph{determined by its spectrum} if any graph having the same spectrum as $G$ is isomorphic to $G$. Let $K_n \setminus P_{\ell}$ be the graph obtained from $K_n$ by removing edges of $P_\ell$, where $P_\ell$ is a path of length $\ell-1$ which is a subgraph of a complete graph $K_n$.  C\'{a}mara and Haemers~\cite{MC} conjectured that $K_n \backslash P_{\ell}$ is determined by its adjacency spectrum for every $2\leq \ell \leq n$. In this paper we show that the conjecture is true for $7\leq \ell \leq9$.\\


\section{Introduction}

All graphs considered in this paper are undirected, finite and simple graphs. For some notations and terminologies in graph spectra, see~\cite{CDS}.

Let $G=(V,E)$ be a graph with vertex set $V(G)=\{v_1,v_2,\cdots,v_n\}$ and edge set $E(G)=\{e_1,e_2,\cdots,e_m\}$. Let $A(G)$ be the (0,1)-adjacency matrix of $G$, \emph{the characteristic polynomial} of $G$ is defined as $P_G(\lambda)=\det(\lambda I-A(G))$. \emph{The spectrum} of $G$ consists of all the eigenvalues of $G$ (including the multiplicities). Two graphs are \emph{cospectral} if they share the same adjacency spectrum. A graph $G$ is said to be \emph{determined by its spectrum} (DS for short) if any graph having the same spectrum as $G$ is necessarily isomorphic to $G$.

The spectrum of a graph encodes useful combinatorial information about the given graph. A fundamental question in the theory of graph spectra is ``Which graphs are DS?". The problem dates back to more than 60 years ago and originates from Chemistry. It has received a lot of attention from researcher in recent years.

However, it turns out that proving a graph to be DS is generally a very hard problem. Up to now, very few classes of graphs with very special structures have been proved to be DS. Usually it is case that the graphs shown to be DS have very few edges, such as the T-shape trees~\cite{WW}, the $\infty$-graphs~\cite{FL}, the lollipop graphs~\cite{WH}, the $\theta$-graphs~\cite{FR}, the graphs with index at most $\sqrt{2+\sqrt{5}}$~\cite{NG}, and the pineapple graphs~\cite{HT}, to just name a few. For dense graphs, it is usually quite difficult to show them to be DS, for example, the complement of the path $\bar{P}_n$ was shown to be DS in~\cite{MD}, but the proof is much more involved than the proof that the path $P_n$ is DS. For some excellent surveys of this topic, we refer the reader to van Dam and Haemers~\cite{DH,DH1} and the references therein.

In~\cite{MC}, C\'{a}mara and Haemers investigated, among others, when a complete graph with some edges deleted is DS. Denoted by $P_\ell$ a path of length $\ell-1$ and $K_n$ the complete graph on $n$ vertices. Denoted by $K_n\setminus P_{\ell}$ the graph obtained from $K_n$ by removing the edges of the path $P_\ell$.
The authors proposed the following  \\

\noindent
\textbf{Conjecture 1 ( C\'{a}mara and Haemers~\cite{MC})}. $K_n\setminus P_{\ell}$ is DS for every $2\leq\ell\leq n$.\\

It was shown in \cite{MC} that Conjecture 1 is true for $\ell\leq 6$. And for $n=\ell$, Conjecture 1 is true and it is the main result from~\cite{MD}. In this paper we show that Conjecture 1 is true for $7\leq\ell\leq9$. Thus we have the following 
\begin{them}\label{main}The graph $K_n\setminus P_{\ell}$ is DS for $7\leq\ell\leq9$.
\end{them}

The proof of the above theorem is based on some eigenvalue properties of the graph $K_n\backslash P_{\ell}$, and a detailed classification of all of its possible cospectral mates.

The rest of the paper is organized as follows. In the next section, we will give some important lemmas that will be needed in the sequel. In Section 3, we present the proof of Theorem~\ref{main}. Conclusions and some further research problems are given in Section 4.


\section{Some lemmas}

In this section, we will present some lemmas which are needed in the proof of the main result. First we give some known results about the spectra of graphs.

\begin{lm}[van Dam and Haemers~\cite{DH}]
The following properties of a graph $G$ can be deduced from the adjacency spectrum:

(i)  The number of vertices.

(ii) The number of edges.

(iii) The number of closed walks of any fixed length.
\end{lm}

Let $N_G(H)$ be the number of subgraphs (not necessarily induced) of a graph $G$ which are isomorphic to $H$ and let $N_G(i)$ be the number of closed walks of length $i$ in $G$. Let $N'_H(i)$ be the number of closed walks of $H$ of length $i$ which contain all the edges of $H$ and let $S_i(G)$ be the set consisting of all the connected subgraph $H$ of $G$ such that $N'_H(i)\neq 0$. It is easy to see that $N_G(i)$ can be expressed by
$$N_G(i)=\sum _{H\in S_i(G)}N_G(H)N'_H(i).$$

\begin{lm}[Omidi~\cite{GR}]
The number of closed walks of length of $2,3,4$ and $5$ of a graph $G$ are given as follows:
\begin{equation*}
  \begin{aligned}
(i) ~N_G(2)&=2m, N_G(3)=6N_G(K_3);\\
(ii)~N_G(4)&=2m+4N_G(P_3)+8N_G(C_4), N_G(5)=30N_G(K_3)+10N_G(C_5)+10N_G(G_a).
  \end{aligned}
\end{equation*}
Where $m$ is the number of edges of $G$ and graph $G_a$ denotes the graph obtained from a triangle by adding a pendent edge to one of its vertices.
\end{lm}

The following lemma is useful which gives the number of triangles of the complement of a graph $G$ in terms of that of $G$ and the numbers of $4$-walks and $5$-walks of $G$.

\begin{lm}[Doob and Haemers~\cite{MD}]
Let $G$ be a graph with $n$ vertices, $m$ edges, $t$ triangles, and degree sequence $d_1,d_2,\cdots,d_n$. Let $\bar{t}$ be the number of triangles in the complement of $G$. Then
\begin{equation*}
    \bar{t}={\left(\begin{array}{c}
  n \\
  3
\end{array}\right)}-(n-1)m+\frac{1}{2}\sum_{i = 1}^n d_i^2-t.
\end{equation*}
\end{lm}

\begin{lm}[C\'{a}mara and Haemers~\cite{MC}]
The number of $4$-walks in the complement of a graph $G$ only depends on the number of vertices and edges of $G$, and the number of different subgraphs (not necessarily induced) in $G$ isomorphic to $P_3, K_2\cup K_2, P_4$ and $C_4$. More precisely, if these numbers are $n, m, m_1, m_2, m_3$, and $m_4$, and $W_n=(n-1)^4+n-1$ is the number of $4$-walks in $K_n$, then the number of $4$-walks in the complement of $G$ equals
 \begin{equation*}
 W_n-(8n^2-32n+34)m+(8n-20)m_1+16m_2-8m_3+8m_4,
  \end{equation*}
\end{lm}
where
 \begin{equation*}
  \begin{aligned}
&n:=\mid V(G)\mid, m:=\mid E(G)\mid,  m_1:=N_G(P_3), \\
&m_2:=N_G(K_2\cup K_2), m_3:=N_G(P_4), m_4:=N_G(C_4).
\end{aligned}
  \end{equation*}

For closed walks of length $5$-walks things become more complicated. We have the following lemma.

\begin{lm}
The number of $5$-walks in the complement of a graph $G$ only depends on the number of vertices and edges of $G$, and the number of different subgraphs (not necessarily induced) in $G$ which are isomorphic to $P_3, K_2\cup K_2, P_4,  K_3, P_3\cup K_2, K_{1,3}, P_5, G_a$ and $C_5$. More precisely, let these numbers be $n, m, m_1, m_2, m_3, s_1, s_2, s_3, s_4, s_5$ and $s_6$, and let $W_n=30{n\choose{3}}+120{n\choose{5}}+30(n-3){n\choose{3}}$ be the number of $5$-walks in $K_n$. Then the number of $5$-walks in the complement of $G$ equals
 \begin{equation*}
  \begin{aligned}
 W_n&-(10n^3-50n^2+90n-60)m+(10n^2-20n)m_1+(40n-120)m_2\\
&-(10n-20)m_3-(30n-60)s_1-20s_2-30s_3+10s_4+10s_5-10s_6,
  \end{aligned}
  \end{equation*}
\end{lm}
where
\begin{equation*}
  \begin{aligned}
& s_1:=N_G(K_3), s_2:=N_G(P_3\cup K_2), s_3:=N_G(K_{1,3}),\\
&s_4:=N_G(P_5), s_5:=N_G(G_a), s_6:=N_G(C_5).
\end{aligned}
  \end{equation*}

\begin{proof} The result is a consequence of the inclusion-exclusion principle. Assume that $G$ and $K_n$ have the same vertex set. Let $E$ be the edge set of $G$. For a subset $F\subset E$. Let $W_F$ denote the set of $5$-walks in $K_n$ containing all edges of $F$. Then the total number of $5$-walks in $K_n$ that contain at least one edge from $E$ equals

\begin{equation*}
\left|\begin{array}{c}
\bigcup\limits_{|F|\geq1}W_F
\end{array}\right|=\sum\limits_{|F|=1}|W_F|-\sum\limits_{|F|=2}|W_F|
+\sum\limits_{|F|=3}|W_F|
-\sum\limits_{|F|=4}|W_F|
+\sum\limits_{|F|=5}|W_F|.
\end{equation*}

If $|F|=1$, then $|W_F|=10(n-2)(n-3)(n-4)+40(n-2)(n-3)+30(n-2)$. If $|F|=2$, then $|W_F|$ depends on the mutual position of the two edges. If they have a vertex in common, then $|W_F|=10(n-3)(n-4)+50(n-3)+30$, and if the two edges are independent then $|W_F|=40(n-4)+40$. If $|F|=3$, then the three edges are $P_4,  K_3, P_3\cup K_2, K_{1,3}$ in $G$ and if the three edges are a path $P_4$ then $|W_F|=10(n-4)+20$, if the three edges are a triangle $K_3$ then $|W_F|=30(n-3)+30$, if the three edges are $P_3\cup K_2,$ then $|W_F|=20$, if the three edges are a star $K_{1,3}$ then $|W_F|=30$. If $|F|=4$, then the four edges are a path or a triangle with a pendant edge to one vertex. If the four edges are a path, then $|W_F|=10$ and if the four edges are a triangle with a pendant edge then $|W_F|=10$. Suppose $|F|=5$, then the edges are a cycle of length $5$ in $G$. Each of them leads to $10$ distinct $5$-walks, so $|W_F|=10$.
\end{proof}

Suppose $K_n\setminus P_{\ell}$ and $K_n\backslash H$ are cospectral. Here and below we define
that $m_i$~(resp. $m_i'$), $s_j$~(resp. $s_j'$) to be the number of $P_3, K_2\cup K_2, P_4, C_4, K_3, P_3\cup K_2, K_{1,3}, P_5, G_a, C_5$ in graph $P_{\ell}$ (resp. graph $H$) for $i=1,2,3,4$, $j=1,2,3,4,5,6$.

  \begin{lm} The pair of graphs $K_n\setminus P_{\ell}$ and $K_n\setminus (C_4\cup P_{\ell-4})$ are not cospectral for any $\ell\geq7$.
\end{lm}
\begin{proof}
For graph $P_{\ell}$ we can directly compute that

\begin{equation*}
\begin{aligned}
m_1=\ell-2, m_2=\frac{(\ell-2)(\ell-3)}{2}, m_3=\ell-3, m_4=0.
 \end{aligned}
\end{equation*}

 And for graph $ C_4\cup P_b$ we have
\begin{equation*}
m'_1=\ell-2, m'_2=\frac{(\ell-2)(\ell-3)}{2}, m'_3=\ell-3, m'_4=1.
\end{equation*}

According to Lemma 2.4, it follows that the pair of graphs in the lemma can be distinguished by the
 number of 4-walks.
\end{proof}

\begin{lm} The pair of graphs $K_n\setminus P_{\ell}$ and $K_n\setminus(aK_{1,3}\cup P_b)$ are not cospectral for any $a\geq1, b\geq2$, where $3a+b=\ell$.
\end{lm}
\begin{proof}
For graph $P_{\ell}$ we can directly compute that

\begin{equation*}
m_1=\ell-2, m_2=\frac{(\ell-2)(\ell-3)}{2}, m_3=\ell-3, m_4=0.
\end{equation*}

And for graph $aK_{1,3}\cup P_2$ and $aK_{1,3}\cup P_3$ we have

\begin{equation*}
m'_1=\ell-2, m'_2=\frac{(\ell-2)(\ell-3)}{2}, m'_3=0, m'_4=0.
\end{equation*}

And for graph $aK_{1,3}\cup P_b$ $(b\geq4)$ we have

\begin{equation*}
m'_1=\ell-2, m'_2=\frac{(\ell-2)(\ell-3)}{2}, m'_3=b-3, m'_4=0.
\end{equation*}

By use of Lemma 2.4 it follows straightforwardly that they can be distinguished by the number of $4$-walks.
\end{proof}

  \begin{lm}The pair of graphs $K_n\setminus P_{\ell}$ and $K_n\setminus (P_a\cup T_{b,c,d})$ are not cospectral for any $a\geq2, b\geq1, c\geq1, d\geq1$, where $a+b+c+d=\ell$ $(\ell\geq6)$.
\end{lm}

\begin{proof} For graph $P_{\ell}$ we can directly compute that

\begin{equation*}
\begin{aligned}
&m_1=\ell-2, m_2=\frac{(\ell-2)(\ell-3)}{2}, m_3=\ell-3, m_4=0,\\
&s_2=(\ell-3)(\ell-4), s_4=\ell-4, s_1=s_3=s_5=s_6=0.
 \end{aligned}
\end{equation*}

 And for graph $P_a\cup T_{b,c,d}$, we can also compute the corresponding number of subgraphs; see Table 1 below:

 \begin{table}[htbp]
 \caption{\small The number of subgraphs in $P_a\cup T_{b,c,d}$}
\centering
\begin{tabular*}{\textwidth}{@{\extracolsep{\fill}}ccccccccccc}
\Xhline{1.5pt}
 \small{$(a,b,c,d)$}& \small{$m_1'$}& \small{$m_2'$}& \small{$m_3'$}& \small{$m_4'$}& \small{$s_1'$}& \small{$s_2'$}& \small{$s_3'$}& \small{$s_4'$}& \small{$s_5'$}& \small{$s_6'$}\\
\Xhline{1pt}
 \small{$a\geq4,b=1,c=1,d=1$}& \small{$\ell-2$}& \small{$\frac{(\ell-2)(\ell-3)}{2}$}& \small{$\ell-6$}& \small{$0$}& & & & & &\\
 \small{$a\geq3,b=1,c=1,d\geq2$}& \small{$\ell-2$}& \small{$\frac{(\ell-2)(\ell-3)}{2}$}& \small{$\ell-5$}& \small{$0$}& & & & & &\\
 \small{$a=2,b=1,c=1,d\geq2$}& \small{$\ell-2$}& \small{$\frac{(\ell-2)(\ell-3)}{2}$}& \small{$\ell-4$}& \small{$0$}& & & & & &\\
 \small{$a\geq3,b=1,c\geq2,d\geq2$}& \small{$\ell-2$}& \small{$\frac{(\ell-2)(\ell-3)}{2}$}& \small{$\ell-4$}& \small{$0$}& & & & & &\\
 \small{$a=2,b\geq2,c\geq2,d\geq2$}& \small{$\ell-2$}& \small{$\frac{(\ell-2)(\ell-3)}{2}$}& \small{$\ell-2$}& \small{$0$}& & & & & &\\
 \small{$a=2,b=1,c=2,d=2$}& \small{$\ell-2$}& \small{$\frac{(\ell-2)(\ell-3)}{2}$}& \small{$\ell-3$}& \small{$0$}& & 9&1 & 1& &\\
 \small{$a=2,b=1,c=2,d\geq3$}& \small{$\ell-2$}& \small{$\frac{(\ell-2)(\ell-3)}{2}$}& \small{$\ell-3$}& \small{$0$}&0 & $\ell^2-7\ell+9$&1 &$\ell-5$&0 &0\\
 \small{$a=2,b=1,c\geq3,d\geq3$}& \small{$\ell-2$}& \small{$\frac{(\ell-2)(\ell-3)}{2}$}& \small{$\ell-3$}& \small{$0$}& 0& $\ell^2-7\ell+9$&1 &$\ell-4$&0 &0\\
 \small{$a=3,b=2,c=2,d=2$}& \small{$\ell-2$}& \small{$\frac{(\ell-2)(\ell-3)}{2}$}& \small{$\ell-3$}& \small{$0$}& 0& $\ell^2-7\ell+9$&1 &$\ell-6$& 0&0\\
 \small{$a\geq4,b=2,c=2,d=2$}& \small{$\ell-2$}& \small{$\frac{(\ell-2)(\ell-3)}{2}$}& \small{$\ell-3$}& \small{$0$}& 0& $\ell^2-7\ell+9$&1 &$\ell-7$& 0&0\\
 \small{$a=3,b=2,c=2,d\geq3$}& \small{$\ell-2$}& \small{$\frac{(\ell-2)(\ell-3)}{2}$}& \small{$\ell-3$}& \small{$0$}&0 & $\ell^2-7\ell+9$&1 &$\ell-5$& 0&0\\
 \small{$a\geq4,b=2,c=2,d\geq3$}& \small{$\ell-2$}& \small{$\frac{(\ell-2)(\ell-3)}{2}$}& \small{$\ell-3$}& \small{$0$}& 0& $\ell^2-7\ell+9$&1 &$\ell-6$& 0&0\\
 \small{$a=3,b=2,c\geq3,d\geq3$}& \small{$\ell-2$}& \small{$\frac{(\ell-2)(\ell-3)}{2}$}& \small{$\ell-3$}& \small{$0$}& 0& $\ell^2-7\ell+9$&1 &$\ell-4$& 0&0\\
 \small{$a\geq4,b=2,c\geq3,d\geq3$}& \small{$\ell-2$}& \small{$\frac{(\ell-2)(\ell-3)}{2}$}& \small{$\ell-3$}& \small{$0$}& 0& $\ell^2-7\ell+9$&1 &$\ell-5$& 0&0\\
 \small{$a=3,b\geq3,c\geq3,d\geq3$}& \small{$\ell-2$}& \small{$\frac{(\ell-2)(\ell-3)}{2}$}& \small{$\ell-3$}& \small{$0$}&0 & $\ell^2-7\ell+9$&1 &$\ell-3$&0 &0\\
 \small{$a\geq4,b\geq3,c\geq3,d\geq3$}& \small{$\ell-2$}& \small{$\frac{(\ell-2)(\ell-3)}{2}$}& \small{$\ell-3$}& \small{$0$}&0 & $\ell^2-7\ell+9$&1 &$\ell-4$&0 &0\\
\Xhline{1.5pt}
\end{tabular*}
\vskip 0.1cm
\end{table}

From Table 1 we know that for only the case $a\geq4, b=2, c=2, d=2$ the pair of graphs in the lemma have the same number of $4$-walks and $5$-walks. However, the adjacency matrix $A$ of $K_n\setminus P_{\ell}$ satisfies ${\rm rank}(A)\geq n-1$ whilst the adjacency matrix $A'$ of $K_n\setminus (P_{\ell-6}\cup T_{2,2,2})$ satisfies ${\rm rank}(A')\leq n-2$. Thus the two graphs have different multiplicities for the eigenvalue 0.

\end{proof}

\begin{lm} Suppose that the pair of graphs $K_n\setminus P_{\ell}$ and $K_n\setminus H$ are cospectral. Then the number of triangles $t'$ in $H$ must be even.
\end{lm}

\begin{proof}
We use the same notations $d_i$ ,$d_i'$, $m_i$, $m_i'$ and $t'$ as above.
If graph $K_n\setminus P_{\ell}$ and $K_n\setminus H$ are cospectral, then they have the same number of $3$-walks. By Lemma 2.3 we have
\begin{equation*}
 \begin{aligned}
\sum_{i = 1}^n d_i'^2-t'=\sum_{i = 1}^n d_i^2.
 \end{aligned}
\end{equation*}
Graphs $K_n\setminus P_{\ell}$ and $K_n\setminus H$ also have the same number of $4$-walks. By Lemma 2.4 we have
\begin{equation*}
 \begin{aligned}
&(8n-20)m'_1+16m'_2-8m'_3+8m'_4\\
=&(8n-20)m_1+16m_2-8m_3+8m_4\\
=&(8n-20)m_1+8(\ell-3)^2.
 \end{aligned}
\end{equation*}

Moreover,
\begin{equation*}
m'_1=\sum_{i = 1}^n{\left(\begin{array}{c}
  d'_i \\
  2
\end{array}\right)}=\sum_{i = 1}^n{\left(\begin{array}{c}
  d_i \\
  2
\end{array}\right)}+t'=m_1+t'.
\end{equation*}

 so we have
 $$(2n-5)t'+4m'_2-2m'_3+2m'_4=2(\ell-3)^2.$$

 As $4m'_2, 2m'_3, 2m'_4$ and $2(\ell-3)^2$ are all even numbers, $2n-5$ is an odd number, so the number of triangles $t'$ in $H$ must be even.
\end{proof}

Let $\Gamma$ be a graph with $|V(\Gamma)|=\ell$, and $C=(C_{ij})_{\ell\times k}$ be a $(0,1)$-matrix.
We construct a new graph, denoted by $(\Gamma,C,n-\ell)$, which is obtained from the disjoint union of $\Gamma$ and $k$ copies of the complete graph $K_{n-\ell}$ by adding some
edges according to the following rule: if $C_{ij}=1$, then each vertex of the $j$-th complete graph $K_{n-\ell}$ is adjacent to vertices $i\in V(\Gamma)$ and is not adjacent to vertices $i\in V(\Gamma)$ with $C_{ij}=0$ (for $i=1,2,\cdots,\ell;j=1,2,\cdots,k$) (see Fig. 1).
 \begin{figure}[htbp]
\centering   
\subfloat{
\begin{minipage}[t]{0.3\textwidth}
   \centering
   \includegraphics[width=2cm,    height=3cm]{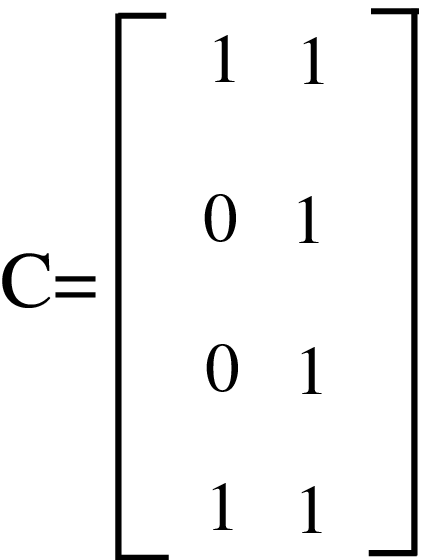}
\end{minipage}
}
\subfloat{
\begin{minipage}[t]{0.3\textwidth}
   \centering
   \includegraphics[width=5cm,    height=3cm]{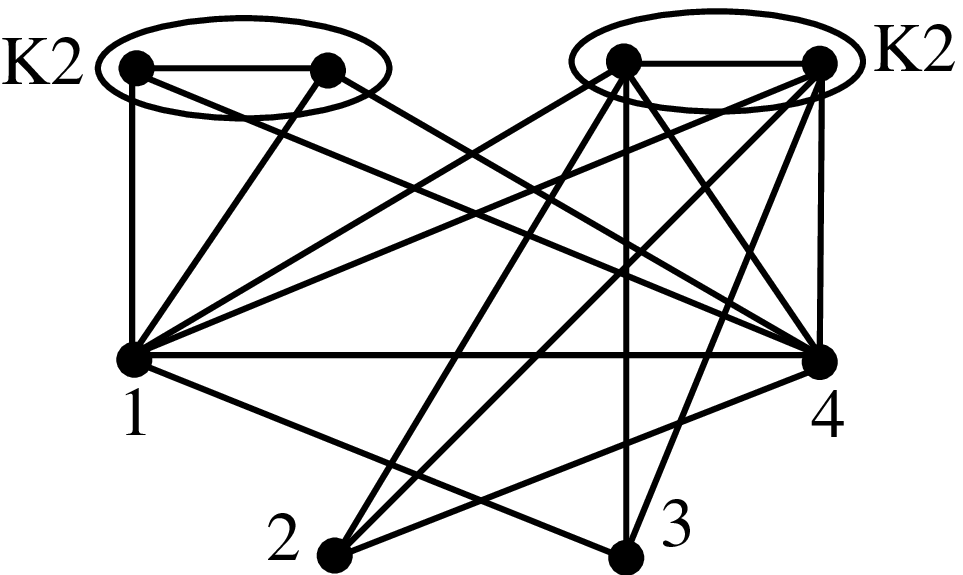}
\end{minipage}
}
 \caption{\small{The graph~$(\Gamma,C,n-l)~$ for $\Gamma=P_4$. }}
\label{fig:ps}
\end{figure}

\begin{lm}[Jing and Koolen~\cite{HJ}]
$\lambda_{min}(\Gamma,C,n-\ell)\geq \lambda_{min}(\Gamma,C,n-\ell+1)$, $\lambda_{min}(\Gamma,C,n-\ell)\geq \lambda_{min}(\mathcal{A}-CC^T)$ and
$lim_{n\rightarrow \infty}\lambda_{min}(\Gamma,C,n-\ell)=\lambda_{min}(\mathcal{A}-CC^T)$, where $\mathcal{A}=A(\Gamma)$.
\end{lm}

A vector $[x_1,\cdots,x_n]$ is called \emph{skew palindromic} if $[x_1,\cdots,x_n]^T=-[x_n,\cdots,x_1]$.

\begin{lm}[Doob and Haemers~\cite{MD}]
Suppose $\bar{A}$ is the adjacency matrix of $K_{\ell}\setminus P_{\ell}$, then $\bar{A}$ has $\lfloor\frac{\ell}{2}\rfloor$ known eigenvalues $\bar{\lambda}_i=-1+2\cos \pi \frac{i-1}{\ell+1}$ with skew palindromic eigenvectors $\xi_i$ if $2\leq i\leq\ell$ and $i\equiv\ell(\emph{mod}~2)$.
\end{lm}

Suppose {\scriptsize $B=\left[\begin{array}{cc}
\bar{A}&J_{\ell\times(n-\ell)}\\
J_{(n-\ell)\times\ell}&J_{(n-\ell)\times (n-\ell)}-I
\end{array}\right]$} is the adjacency matrix of $K_{n}\setminus P_{\ell}$. Then $B$ also has $\lfloor\frac{\ell}{2}\rfloor$ known eigenvalues $\bar{\lambda}_i=-1+2\cos \pi \frac{i-1}{\ell+1}$ with skew palindromic eigenvectors $[\xi_i,\emph{0}]$ since $\xi_i$ is a skew palindromic eigenvector, orthogonal to the all-one vector \textbf{1}.

\begin{lm}[Doob and Haemers~\cite{MD}]
If $\ell>2$, then every eigenvalue of $K_{\ell}\setminus P_{\ell}$ has multiplicity one.
\end{lm}

The following lemma lies at the heart of the proof of Theorem~\ref{main}.
\begin{lm} Let $n>2$. Then every eigenvalue of graph $K_n\setminus P_{\ell}$ has multiplicity one, except for -1. The multiplicity of $-1$ is $n-\ell$ if $\ell$ is odd, and it is  $n-\ell-1$ if $\ell$ is even. Moreover, we have $\lambda_{min}(K_n\setminus P_{\ell})>-3$.
\end{lm}

\begin{proof}
First we give the eigenvectors associated with eigenvalue -1 explicitly. We distinguish the following cases:\\

\noindent
\textbf{Case 1}. For $\ell=2k$, there are $n-\ell-1$ eigenvectors:
\begin{equation*}
 \begin{aligned}
&(\overbrace{0,0,\cdots,0}^{\ell},-1,1,0,\cdots,0,0),\\
&(0,0,\cdots,0,-1,0,1,\cdots,0,0),\\
&~~~~~~~~~~~~~\vdots\\
&(0,0,\cdots,0,-1,0,0,\cdots,1,0),\\
&(0,0,\cdots,0,-1,0,0,\cdots,0,1).\\
 \end{aligned}
\end{equation*}

\noindent
\textbf{Case 2}.
For $\ell=4k+1$, there are $n-\ell$ eigenvectors:
\begin{equation*}
 \begin{aligned}
&(\overbrace{-1,0,1,0,-1,0,1,0,\cdots,-1,0,1,0,-1}^{\ell},1,0,\cdots,0,0),\\
&(-1,0,1,0,-1,0,1,0,\cdots,-1,0,1,0,-1,0,1,\cdots,0,0),\\
&~~~~~~~~~~~~~\vdots\\
&(-1,0,1,0,-1,0,1,0,\cdots,-1,0,1,0,-1,0,0,\cdots,1,0),\\
&(-1,0,1,0,-1,0,1,0,\cdots,-1,0,1,0,-1,0,0,\cdots,0,1).\\
 \end{aligned}
\end{equation*}

\noindent
\textbf{Case 3}.
For $\ell=4k+3$, there are $n-\ell$ eigenvectors:
\begin{equation*}
 \begin{aligned}
 &(\overbrace{-1,0,1,0,-1,0,1,0,\cdots,-1,0,1,0,-1,0,1}^{\ell},0,0,\cdots,0),\\
&(\overbrace{0,0,\cdots,0}^{\ell}, -1,1,0,\cdots,0,0),\\
&(0,0,\cdots,0, -1,0,1,\cdots,0,0),\\
&~~~~~~~~~~~~~\vdots\\
&(0,0,\cdots,0, -1,0,0,\cdots,1,0),\\
&(0,0,\cdots,0, -1,0,0,\cdots,0,1).\\
 \end{aligned}
\end{equation*}

It follows from Lemma 2.10 that
$$\lambda_{min}(K_n\setminus P_{\ell})\geq \lambda_{min}(\mathcal{A}-CC^T)=\lambda_{min}(J-I-A-J)=\lambda_{min}(-I-A),$$
where $A$ is the adjacency matrix of the path $P_{\ell}$.

The eigenvalues of $P_{\ell}$ (that is the eigenvalues of $A$) are well known to be
$$\lambda_i=2\cos \frac{\pi i}{\ell+1}, i=1,2,\cdots,\ell.$$

So the smallest eigenvalue of $K_n\setminus P_{\ell}$ satisfies that $\lambda_{min}(K_n\setminus P_{\ell})>-3$.

It remains to show that every eigenvalue of graph $K_n\setminus P_{\ell}$ has multiplicity one, except for -1. We prove this assertion by induction on $k$. First suppose $k=1$, we shall show that $K_{\ell+1}\setminus P_{\ell}$ has multiplicity one.

From Lemma 2.12 $K_{\ell}\setminus P_{\ell}$ has $\ell$ different eigenvalues which will be ordered as $\lambda_1>\lambda_2>\cdots>\lambda_{\ell}$ where $\lfloor\frac{\ell}{2}\rfloor$ known eigenvalues are given in lemma 2.11. $K_{\ell+1}\setminus P_{\ell}$ has $\ell+1$ eigenvalues which will be ordered as $\mu_1\geq\mu_2\geq\cdots\geq\mu_{\ell+1}$. The eigenvalues of  $K_{\ell}\setminus P_{\ell}$ interlace those of $K_{\ell+1}\setminus P_{\ell}$, that is, $\mu_i\geq\lambda_i\geq\mu_{i+1}$ for $i=1,2,\cdots,\ell$. Next we shall show that $\lfloor\frac{\ell}{2}\rfloor$ known eigenvalues have multiplicity 1. This proves this assertion, since every other eigenvalue lies between two eigenvalues with multiplicity 1.

Suppose $\bar{\lambda}$ is such an eigenvalue. Then, substituting $i=\ell+2-2m$ gives $\bar{\lambda}=-1-2\cos\varphi$ with $\varphi=2\pi m/(\ell+1)$ for some integer $m$, $1\leq m\leq\ell/2$. Let $x=[x_1,x_2,\cdots, x_{\ell}]^T$, $y_1=[x_1,x_2,\cdots, x_{\ell},a]^T$ be an eigenvector for $\bar{\lambda}$. If $a=0$, $y_1$ is an eigenvector of $K_{\ell}\setminus P_{\ell}$ and the corresponding eigenvalue is $2\cos\varphi$. Suppose $a\neq0$, $B_1y_1=\bar{\lambda}y_1$ (where $B_1=\left[\begin{array}{cc}
J-I-A&\textbf{1}\\
\textbf{1}^T&0
\end{array}\right]$ is the adjacency matrix of $K_{\ell+1}\setminus P_{\ell}$ ) implies that
\begin{equation*}
 \begin{aligned}
&\textbf{1}^Tx=\bar{\lambda}a,\\
&x_2=a(\bar{\lambda}+1)-(\bar{\lambda}+1)x_1,\\
&x_i=a(\bar{\lambda}+1)-(\bar{\lambda}+1)x_{i-1}-x_{i-2},~{\rm for}~i=2,3,\ldots, \ell.
 \end{aligned}
\end{equation*}

The general solution of this recurrence has the form

$$x_i=\alpha \cos i\varphi+\beta \sin i\varphi+\frac{a(\bar{\lambda}+1)}{(\bar{\lambda}+3)}.$$

(Note that $\bar{\lambda}+3>0$). Substituting $x_0=0$ gives $\alpha=-\frac{a(\bar{\lambda}+1)}{(\bar{\lambda}+3)}$. Moreover, $x_{\ell}=\alpha \cos\ell\varphi+\beta \sin\ell\varphi+\frac{a(\bar{\lambda}+1)}{(\bar{\lambda}+3)}=\alpha \cos\varphi-\beta \sin\varphi+\frac{a(\bar{\lambda}+1)}{(\bar{\lambda}+3)}$. Hence
$$x_1+x_{\ell}=2\alpha\cos \varphi+\frac{2a(\bar{\lambda}+1)}{(\bar{\lambda}+3)}=a(\bar{\lambda}+1).$$

Next we equate $\textbf{1}^T(\bar{A}x)$ to $(\textbf{1}^T\bar{A})x$ to get $\bar{\lambda}^2-\ell=(\ell-3)\bar{\lambda}+(\bar{\lambda}+1)$. Thus we find $\ell=\bar{\lambda}+1-\frac{2}{\bar{\lambda}+1}$ which cannot be an integer; a contradiction.

Next, suppose $k\geq2$. By induction, we have $K_{\ell+k}\setminus P_{\ell}$ has different eigenvalues (except for -1) which will be ordered as $\lambda_1>\lambda_2>\cdots>-1=-1=\cdots=-1>\cdots>\lambda_{\ell+k}$ where $\lfloor\frac{\ell}{2}\rfloor$ known eigenvalues are given in lemma 2.11. $K_{\ell+k+1}\setminus P_{\ell}$ has $\ell+k+1$ eigenvalues which will be ordered as $\mu_1\geq\mu_2\geq\cdots\geq\mu_{\ell+k+1}$. The eigenvalues of  $K_{\ell+k}\setminus P_{\ell}$ interlace those of $K_{\ell+k+1}\setminus P_{\ell}$, that is, $\mu_i\geq\lambda_i\geq\mu_{i+1}$ for $i=1,2,\cdots,\ell+k$. Next we shall show that $\lfloor\frac{\ell}{2}\rfloor$ known eigenvalues have multiplicity 1.

Suppose $\bar{\lambda}$ is such an eigenvalue. Let $x=[x_1,x_2,\cdots, x_{\ell}]^T$,$\gamma=[a_1,a_2,\cdots, a_{k+1}]^T$ and $y_2=[x_1,x_2,\cdots, x_{\ell},a_1,\cdots,a_{k+1}]^T$ be an eigenvector for $\bar{\lambda}$. Actually $a_1,a_2,\cdots,a_{k+1}$ must be equal if $\bar{\lambda}\neq -1$.

Let $B_2=\left[\begin{array}{cc}
J_1-I-A&J_2\\
J_2^T&J_3-I
\end{array}\right]$ be the adjacency matrix of $K_{\ell+k+1}\setminus P_{\ell}$. It follows from $B_2y_2=\bar{\lambda}y_2$ that
 \begin{equation*}
 \begin{aligned}
J_2^Tx+(J_3-I)\gamma=\bar{\lambda}\gamma,
 \end{aligned}
\end{equation*}
 i.e.,

 $$[(\bar{\lambda}+1)I-J_3]\gamma=J_2^Tx.$$

 Then we have $(\bar{\lambda}+1)(a_i-a_j)=0$ and if $\bar{\lambda}\neq -1$, then $a_1,a_2,\cdots,a_{k+1}$ must be equal to each other.

 So we can assume $y_2=[x_1,x_2,\cdots, x_{\ell},a,\cdots,a]^T$. If $a=0$, then $y_2$ is an eigenvector of $K_{\ell+k+1}\setminus P_{\ell}$. Suppose $a\neq0$. It follows from $B_2y_2=\bar{\lambda}y_2$ that
 \begin{equation*}
  \left\{\begin{aligned}
(J_1-I-A)x+(k+1)a\textbf{1}=\bar{\lambda}x,\\
\textbf{1}^Tx+ka=\bar{\lambda}a.
 \end{aligned}
 \right.
\end{equation*}
i.e.,
\begin{equation*}
 \begin{aligned}
x_i=a(\bar{\lambda}+1)-(\bar{\lambda}+1)x_{i-1}-x_{i-2},~{\rm for}~i=2,3,\ldots, \ell.
 \end{aligned}
\end{equation*}

 Similarly as before we also get a contradiction. This completes the proof.
\end{proof}

\noindent
\textbf{Remark}. As we know that every eigenvalue of $K_{\ell}\setminus P_{\ell-1}$ has multiplicity one. So all conclusions in \cite{MC} hold for the graph $K_{\ell}\setminus P_{\ell-1}$. So we directly have $K_{\ell}\setminus P_{\ell-1}$ is DS.

\begin{lm}
Suppose $K_n\setminus H$ is cospectral with $K_n\setminus P_{\ell}$. Then graph $H$ has the following three properties.
\begin{enumerate}
\item No component of $H$ is a cycle except for $C_3$ and $C_4$.
\item Graph $H$ cannot contain the disjoint union of two cycles $C_k\cup C_s$, where $C_k$ and $C_s$
are both induced subgraph of $H$.
\item No two components of $H$ are paths of the same nonzero length except for $P_3$.
\end{enumerate}
\end{lm}

 \begin{proof}
 (i) Suppose that there exists one component of $H$ that is a cycle $C_a~ (a\neq 3,4)$. The eigenvalues of $C_a$ are 2cos$(2\pi i/a)$ for $i=0,1,\cdots,a-1$. Note that $a-i$ and $i$ give the same value, so almost all eigenvalues have multiplicity 2.  Since each cycle has eigenvalue $2$ with all-one eigenvector \textbf{1}, every other eigenvectors $\xi_i$ with eigenvalue 2cos$(2\pi i/a)$ are all orthogonal to  all-one vector \textbf{1}, for $i=1,\cdots,a-1$. We readily find that $K_n\setminus H$ has eigenvalues $-1-2\cos(2\pi i/a)$  with eigenvectors $[\xi_i,o]^T$ for $i=1,\cdots,a-1$, which contradicts Lemma 2.13.

 (ii) Suppose $H$ contains an induced subgraph $\Delta$ which consists of two disjoint cycles $C_k$ and $C_s$. Then the complement of $\Delta$ has eigenvalue -3 with the eigenvector $v=[\overbrace{-s,\cdots,-s}^k,\overbrace{k,\cdots,k}^s]$. Eigenvalue interlacing theorem gives that $\bar{\lambda}_{min}(K_n\setminus P_{\ell})\leq-3$; a contradiction.

 (iii) Suppose $H$ contains $P_{k}$ ($k\neq3$) twice. Then both paths have an eigenvalue $\lambda$ with a skew palindromic eigenvector. Since a skew palindromic vector is orthogonal to \textbf{1}, graph $H$ has an eigenvalue $\lambda$ with at least two independent eigenvectors orthogonal to \textbf{1}. This implies that $K_n\setminus H$ has an eigenvalue $-\lambda-1$ with multiplicity at least two, which contradicts Lemma 2.13.
 \end{proof}

\section{The Proof of Theorem~\ref{main}}

In this section, we present the proof of Theorem~\ref{main}.

Let $G=P_{\ell}+(n-\ell)K_1$ (the complement of $K_n\backslash P_{\ell}$). Apparently, $G$ has $n$ vertices, $\ell-1$ edges, no triangles, and degree sequence $\{0^{n-\ell},1^2,2^{\ell-2}\}$. Thus, the number of triangles $\bar{t}$ in graph $K_n\setminus P_{\ell}$ is
\begin{equation*}
  \begin{aligned}
 \bar{t}={n\choose{3}}-(n-1)(\ell-1)+\frac{1}{2}(4\ell-6).
  \end{aligned}
\end{equation*}

 Let $\Gamma=H+(n-V(H))K_1$ (the complement of $K_n\backslash H$), which has $n$ vertices, $\ell-1$ edges, $t'$ triangles, and degree sequence $\{0^{n-\sum_{i=1}^kx_i},1^{x_1},2^{x_2},\cdots, k^{x_k}\}$. Then the number of triangles $\bar{t}'$ in graph $K_n\backslash H$ is

 \begin{equation*}
  \begin{aligned}
 \bar{t}'&={n\choose{3}}-(n-1)(\ell-1)+\frac{1}{2}\sum_{i=1}^k(i^2x_i)-t'.
  \end{aligned}
\end{equation*}

Suppose the pair of graphs $K_n\backslash H$ and $K_n\setminus P_{\ell}$ are cospectral, then they have the same number of triangles. So we have

 \begin{equation}
 \left\{ \begin{aligned}
\sum_{i=1}^kix_i & = 2l-2, \\
\sum_{i=1}^ki^2x_i& =4l-6+2t'.
 \end{aligned}
 \right.
  \end{equation}

\begin{lemma} \label{P7} The graph $K_n\setminus P_7$ is determined by its adjacency spectrum.
\end{lemma}
\begin{proof} If $\ell=7$, then we have $0\leq t'\leq4, 0\leq k\leq5,0\leq x_i\leq min\{\frac{12}{i},\frac{22+2t'}{i^2}\}$, with the help of Mathematica software, we find all possible combinations of $\{t',x_1,x_2,x_3,x_4,x_5\}$ that satisfy (1):

 \begin{equation*}
  \begin{aligned}
&\{0,2,5,0,0,0\},
\{0,5,2,1,0,0\},
\{1,0,6,0,0,0\},
\{1,3,3,1,0,0\},
\{1,6,0,2,0,0\},\\
&\{1,8,0,0,1,0\},
\{2,1,4,1,0,0\},
\{2,4,1,2,0,0\},
\{2,6,1,0,1,0\},
\{3,2,2,2,0,0\},\\
&\{3,4,2,0,1,0\},
\{4,0,3,2,0,0\},
\{4,2,3,0,1,0\},
\{4,3,0,3,0,0\},
\{4,5,0,1,1,0\}.
 \end{aligned}
  \end{equation*}

A set of parameters $\{t',x_1,x_2,x_3,x_4,x_5\}$ is called \emph{graphic} if there exists a graph with the same parameters. Actually not all of these combinations are graphic and for some of them there may exist more than one graphs; see the Table 2. Here we only give the graphic combinations and corresponding graphs (the related subgraph in the Table see Fig. 2):

\begin{table}[htbp]
\caption{\small All the possible graphs cospectral with $K_n\setminus P_7$}
\medskip
\centering
\begin{tabular*}{\textwidth}{@{\extracolsep{\fill}}cc}
\Xhline{1.5pt}
\small{graphic combinations}&\small{corresponding graphs}\\
\Xhline{1pt}
\small{\{0,2,5,0,0,0\}}&\small{ ~$C_5\cup P_2,C_4\cup P_3~$}\\

\small{\{0,5,2,1,0,0\}}&\small{~$T_{1,1,3}\cup P_2,T_{1,2,2}\cup P_2,T_{1,1,2}\cup P_3, T_{1,1,1}\cup P_4~$}\\

\small{\{1,3,3,1,0,0\}}&\small{~$K_{1,3}\cup C_3,G_a\cup P_3,G_h\cup P_2~$}\\
\Xhline{1.5pt}
\end{tabular*}
\vskip 0.2cm
\end{table}

By Lemma 2.6, Lemma 2.7 and Lemma 2.9, we have that graph $K_n\setminus P_7$ is determined by its adjacency spectrum.
\end{proof}

 \begin{figure}[htbp]
\centering   
\subfloat[$G_a$]{
\begin{minipage}[t]{0.15\textwidth}
   \centering
   \includegraphics[height=0.9cm,width=1.5cm]{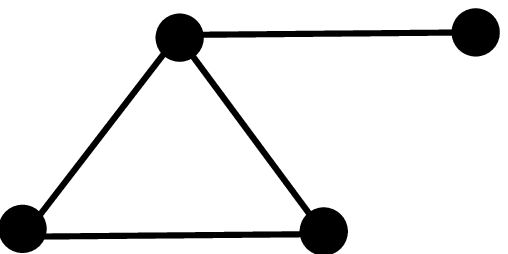}
\end{minipage}
}
\subfloat[$G_b$]{
\begin{minipage}[t]{0.15\textwidth}
   \centering
   \includegraphics[height=0.8cm,width=2cm]{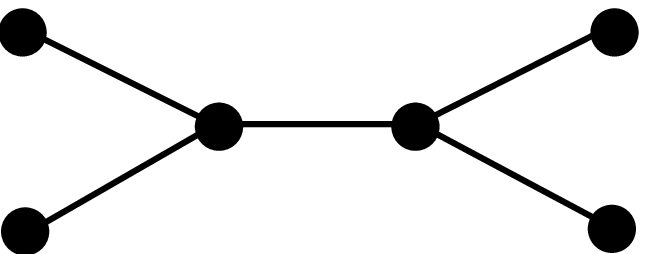}
\end{minipage}
}
\subfloat[$G_c$]{
\begin{minipage}[t]{0.15\textwidth}
   \centering
   \includegraphics[height=0.9cm,width=2cm]{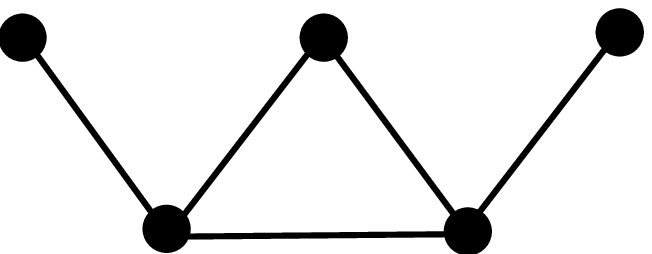}
\end{minipage}
}
\subfloat[$G_d$]{
\begin{minipage}[t]{0.15\textwidth}
   \centering
   \includegraphics[height=0.9cm,width=1.5cm]{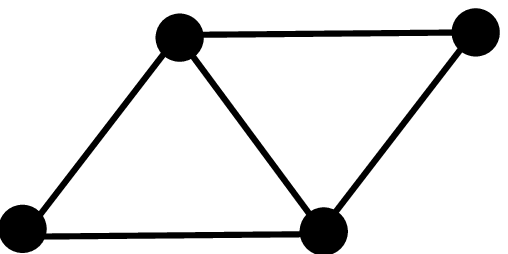}
\end{minipage}
}
\subfloat[$G_e$]{
\begin{minipage}[t]{0.15\textwidth}
   \centering
   \includegraphics[height=0.9cm,width=2.6cm]{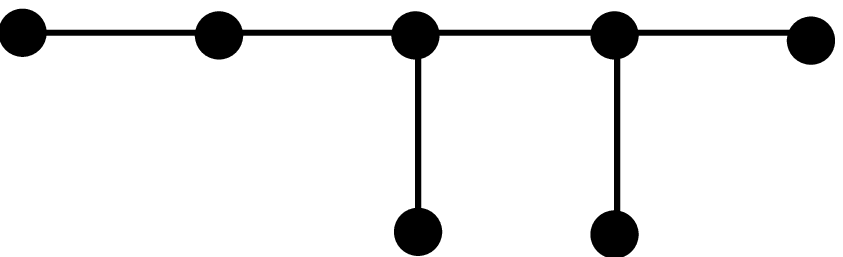}
\end{minipage}
}\\
\subfloat[$G_f$]{
\begin{minipage}[t]{0.15\textwidth}
   \centering
   \includegraphics[height=0.9cm,width=2.6cm]{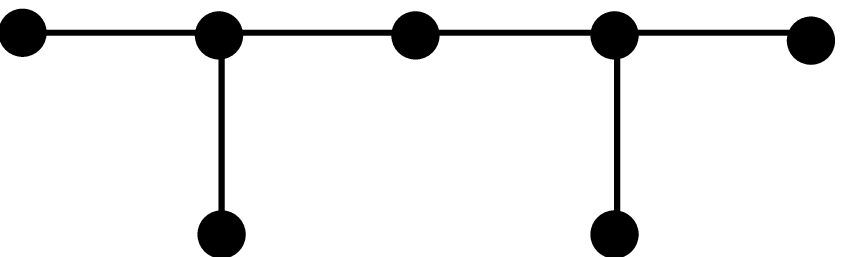}
\end{minipage}
}
\subfloat[$G_g$]{
\begin{minipage}[t]{0.15\textwidth}
   \centering
   \includegraphics[height=1.2cm,width=2cm]{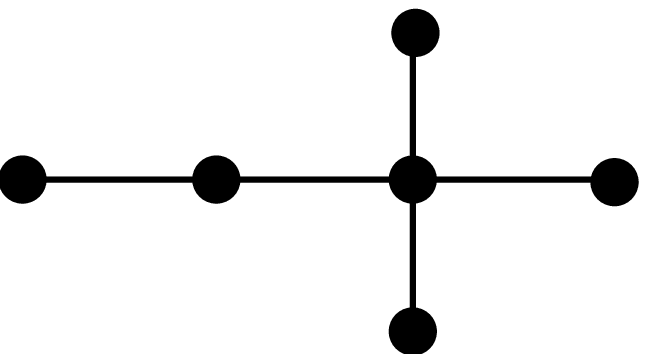}
\end{minipage}
}
\subfloat[$G_h$]{
\begin{minipage}[t]{0.15\textwidth}
   \centering
   \includegraphics[height=0.9cm,width=1.5cm]{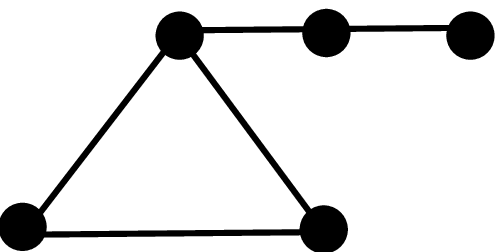}
\end{minipage}
}
\subfloat[$G_i$]{
\begin{minipage}[t]{0.15\textwidth}
   \centering
   \includegraphics[height=0.9cm,width=2cm]{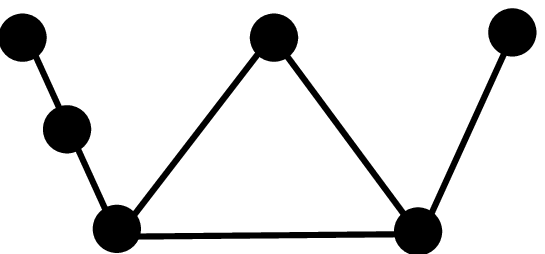}
\end{minipage}
}
\subfloat[$G_j$]{
\begin{minipage}[t]{0.15\textwidth}
   \centering
   \includegraphics[height=0.9cm,width=1.4cm]{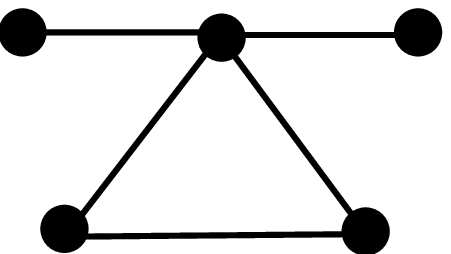}
\end{minipage}
}
\caption{Subgraphs related to $K_n\setminus P_7,K_n\setminus P_8, {\rm and}~ K_n\setminus P_9$}
\label{fig:ps}
\end{figure}

\vspace{1.2cm}

\begin{lemma}\label{P8} The graph $K_n\setminus P_8$ is determined by its adjacency spectrum.
\end{lemma}

\begin{proof} If $\ell=8$, then we have $0\leq t'\leq4,0\leq k\leq5,0\leq x_i\leq min\{\frac{14}{i},\frac{26+2t'}{i^2}\}$, similarly we find all possible combinations of $\{t',x_1,x_2,x_3,x_4,x_5\}$ that satisfy (1):
 \begin{equation*}
  \begin{aligned}
&\{0,2,6,0,0,0\},
\{0,5,3,1,0,0\},
\{0,8,0,2,0,0\},
\{0,10,0,0,1,0\},
\{1,0,7,0,0,0\},\\
&\{1,3,4,1,0,0\},
\{1,6,1,2,0,0\},
\{1,8,1,0,1,0\},
\{2,1,5,1,0,0\},
\{2,4,2,2,0,0\},\\
&\{2,6,2,0,1,0\},
\{3,2,3,2,0,0\},
\{3,4,3,0,1,0\},
\{3,5,0,3,0,0\},
\{3,7,0,1,1,0\},\\
&\{4,0,4,2,0,0\},
\{4,2,4,0,1,0\},
\{4,3,1,3,0,0\},
\{4,5,1,1,1,0\},
\{4,9,0,0,0,1\}.
 \end{aligned}
  \end{equation*}

  Table 3 only gives the graphic combinations and its corresponding graphs (for the related subgraph in the table, see Fig. 2):
 \begin{table}[htbp]
\caption{\small possible cospectral graphs with $K_n\setminus P_8$}
\centering
\begin{tabular*}{\textwidth}{@{\extracolsep{\fill}}cc}
\Xhline{1.5pt}
\small{graphic combinations}&\small{corresponding graphs}\\
\Xhline{1pt}
 \small{\{0,2,6,0,0,0\}}& \small{ ~$C_6\cup P_2,C_5\cup P_3,C_4\cup P_4,~$}\\
  \small{\{0,5,3,1,0,0\}}& \small{~$T_{1,1,4}\cup P_2,T_{1,2,3}\cup P_2,T_{2,2,2}\cup P_2,T_{1,1,3}\cup P_3, T_{1,2,2}\cup P_3,~$~$T_{1,1,2}\cup P_4, T_{1,1,1}\cup P_5~$}\\
  \small{\{0,8,0,2,0,0\}}& \small{~$G_b\cup 2P_2,2K_{1,3}\cup P_2~$}\\
  \small{\{0,10,0,0,1,0\}}& \small{~$K_{1,4}\cup 3P_2~$}\\
  \small{\{1,0,7,0,0,0\}}&  \small{~$C_3\cup C_4~$}\\
  \small{\{1,3,4,1,0,0\}}&  \small{~$G_a\cup P_4,T_{1,1,2}\cup C_3~$}\\
  \small{\{1,6,1,2,0,0\}}&  \small{~$G_c\cup 2P_2~$}\\
  \small{\{2,1,5,1,0,0\}}& \small{~$G_a\cup C_3~$}\\
  \small{\{2,4,2,2,0,0\}}& \small{ ~$G_d\cup 2P_2~$}\\
\Xhline{1.5pt}
\end{tabular*}
\vskip 0.1cm
\end{table}

By Lemma 2.6, lemma 2.7 , Lemma 2.8, Lemma 2.9, Lemma 2.14 all these graphs are not cospectral with graph $K_n\backslash P_8$. So graph $K_n\backslash P_8$ is $DS$.
\end{proof}

\begin{lemma}\label{P9} The graph $K_n\setminus P_9$ is determined by its adjacency spectrum.
\end{lemma}
\begin{proof}
 If $\ell=9$, we have $0\leq t'\leq5,0\leq k\leq6,0\leq x_i\leq min\{\frac{16}{i},\frac{30+2t'}{i^2}\}$, we find all possible combinations of $\{t',x_1,x_2,x_3,x_4,x_5,x_6\}$ that satisfy (1):
  \begin{equation*}
  \begin{aligned}
&\{0,0,6,0,1,0,0\},
\{0,1,3,3,0,0,0\},
\{0,3,3,1,1,0,0\},
\{0,4,0,4,0,0,0\},
\{0,6,0,2,1,0,0\},\\
&\{0,7,2,0,0,1,0\},
\{0,8,0,0,2,0,0\},
\{1,1,4,1,1,0,0\},
\{1,2,1,4,0,0,0\},
\{1,4,1,2,1,0,0\},\\
&\{1,5,3,0,0,1,0\},
\{1,6,1,0,2,0,0\},
\{1,8,0,1,0,1,0\},
\{2,0,2,4,0,0,0\},
\{2,2,2,2,1,0,0\},\\
&\{2,3,4,0,0,1,0\},
\{2,4,2,0,2,0,0\},
\{2,6,1,1,0,1,0\},
\{3,0,3,2,1,0,0\},
\{3,1,5,0,0,1,0\},\\
&\{3,2,3,0,2,0,0\},
\{3,3,0,3,1,0,0\},
\{3,4,2,1,0,1,0\},
\{3,5,0,1,2,0,0\},
\{3,10,0,0,0,0,1\},\\
&\{4,0,4,0,2,0,0\},
\{4,1,1,3,1,0,0\},
\{4,2,3,1,0,1,0\},
\{4,3,1,1,2,0,0\},
\{4,5,0,2,0,1,0\},\\
&\{4,7,0,0,1,1,0\},
\{4,8,1,0,0,0,1\},
\{5,0,4,1,0,1,0\},
\{5,1,2,1,2,0,0\},
\{5,3,1,2,0,1,0\},\\
&\{5,5,1,0,1,1,0\},
\{5,6,2,0,0,0,1\}.
 \end{aligned}
  \end{equation*}

Table 4 only gives graphic combinations and its corresponding graphs (for the related subgraph in the table, see Fig. 2):

\begin{table}
\caption{\small possible cospectral graphs with $K_n\setminus P_9$}
\centering
\begin{tabular*}{\textwidth}{@{\extracolsep{\fill}}cc}
\Xhline{1.5pt}
\small{graphic combinations}&\small{corresponding graphs}\\
\Xhline{1pt}
\small{\{0,2,7,0,0,0,0\}}&\small{ ~$C_7\cup P_2,C_6\cup P_3,C_5\cup P_4,C_4\cup P_5,~$}\\
\hdashline[1pt/1pt]
  \small{\{0,5,4,1,0,0,0\}}& \small{~$T_{1,1,5}\cup P_2,T_{1,2,4}\cup P_2,T_{1,3,3}\cup P_2,T_{2,2,3}\cup P_2,T_{1,1,4}\cup P_3,T_{1,2,3}\cup P_3,~$}\\
   &\small{~$T_{2,2,2}\cup P_3,T_{1,1,3}\cup P_4,T_{1,2,2}\cup P_4, T_{1,1,2}\cup P_5, T_{1,1,6}\cup P_6~$}\\
\hdashline[1pt/1pt]
  \small{\{0,8,1,2,0,0,0\}}&\small{~$G_e\cup 2P_2,G_f\cup 2P_2,G_b\cup P_2\cup P_3,T_{1,1,2}\cup K_{1,3}\cup P_2,2K_{1,3}\cup P_3~$}\\
  \hdashline[1pt/1pt]
  \small{\{0,10,1,0,1,0,0\}}&\small{~$G_g\cup 3P_2,K_{1,4}\cup 2P_2\cup P_3~$}\\
  \hdashline[1pt/1pt]
  \small{\{1,0,8,0,0,0,0\}}& \small{~$C_3\cup C_5~$}\\
  \hdashline[1pt/1pt]
  \small{\{1,3,5,1,0,0,0\}}& \small{~$G_h\cup P_4,G_a\cup P_5~$}\\
  \hdashline[1pt/1pt]
  \small{\{1,6,2,2,0,0,0\}}& \small{~$G_i\cup 2P_2,G_c\cup P_2\cup P_3~$}\\
  \hdashline[1pt/1pt]
  \small{\{1,8,2,0,1,0,0,0\}}&\small{ ~$G_j\cup 3P_2~$}\\
  \hdashline[1pt/1pt]
  \small{\{2,1,6,1,0,0,0\}}&\small{~$G_h\cup C_3~$}\\
  \hdashline[1pt/1pt]
  \small{\{2,4,3,2,0,0,0\}}& \small{~$G_d\cup P_2\cup P_3~$}\\
\Xhline{1.5pt}
\end{tabular*}
\vskip 0.1cm
\end{table}

Similarly, by Lemmas 2.6~2.9 and Lemma 2.14, all these graphs are not cospectral with graph $K_n\backslash P_9$, except for graphs $G_b\cup P_2\cup P_3$, $T_{1,1,2}\cup K_{1,3}\cup P_2$,and $G_d\cup P_2\cup P_3$.  However, these three graphs have different number of $4$-walks with $K_n\backslash P_9$. So graph $K_n\backslash P_9$ is DS.
 \end{proof}

 Finally, we present the proof Theorem~\ref{main}:

 \begin{proof}Combining Lemmas~\ref{P7}, \ref{P8} and \ref{P9}, Theorem~\ref{main}follows immediately.
 \end{proof}

\section{Conclusions}

In this paper, we have derived some eigenvalue properties for the graph $K_n\backslash P_{\ell}$, based on which we are able to show that $K_n\backslash P_{\ell}$ is DS for some small values of $\ell$. It is noticed that the multiplicity of the eigenvalue -1 of $K_n\backslash P_{\ell}$ is larger than one, while it is one for $K_{\ell}\backslash P_{\ell}$. So the results in~\cite{MD} cannot be directly applied in this paper. Also, the proof of Theorem~\ref{main} is based on a detailed classification of all the possible cospectral graphs of $K_n\backslash P_{\ell}$, and it is geting more involved for larger $\ell$. Thus, to deal with the general case of Conjecture 1, new tools and insights are needed. This will be investigated in the future.

\end{document}